\documentclass[a4paper,10pt]{amsart}
\makeatletter
\newcommand{\rmnum}[1]{\romannumeral #1}
\newcommand{\Rmnum}[1]{\expandafter\@slowromancap\romannumeral #1@}
\makeatother
\usepackage{hyperref}
\usepackage{mathrsfs}
\usepackage{amsfonts}
\usepackage{txfonts}
\usepackage{amssymb}
\usepackage[arrow,matrix]{xy}
\usepackage{amsmath,amssymb,amscd,bbm,amsthm,mathrsfs,dsfont}
\usepackage{amsmath,amscd}\input amssym.def
\usepackage{amsfonts,amssymb}
\usepackage{amsmath,amscd}\input amssym.def
\newtheorem{theorem}{Theorem}[section]
\newtheorem{lemma}[theorem]{Lemma}
\theoremstyle{definition}
\newtheorem{defn}[theorem]{Definition}
\newtheorem{exa}[theorem]{Example}

\newtheorem{corollary}[theorem]{Corollary}
\newtheorem{prop}[theorem]{Proposition}
\theoremstyle{remark}

\numberwithin{equation}{section}

\begin{document}
\title[differential graded Poisson Hopf algebras]{\bf The structures on the universal enveloping algebras of differential graded Poisson Hopf algebras}
\author{MentTian Guo}
\address{Guo: Department of Mathematics, Zhejiang Normal University,
Jinhua, Zhejiang, 321004 P.R. China} \email{1448199409@qq.com}
\author{Xianguo Hu}
\address{Hu: Department of Mathematics, Zhejiang Normal University, Jinhua, Zhejiang, 321004 P.R. China}
\email{2015210420@zjnu.edu.cn}
\author{Jiafeng L\"u$^*$}
\address{L\"u: Department of Mathematics, Zhejiang Normal University, Jinhua, Zhejiang 321004, P.R. China}
\email{jiafenglv@zjnu.edu.cn}

\author{Xingting Wang}
\address{Wang: Department of Mathematics, Temple University, Philadelphia 19122, USA}
\email{xingting@temple.edu}

\thanks{This work was supported by National Natural Science Foundation of China (Grant No.11571316) and Natural
Science Foundation of Zhejiang Province (Grant No. LY16A010003).}
\thanks{*corresponding author}
\maketitle
\begin{abstract} In this paper, the so-called differential graded (DG for short) Poisson
Hopf algebra is introduced, which can be considered as a natural extension of
Poisson Hopf algebras in the differential graded setting. The structures on the universal enveloping algebras of differential graded Poisson Hopf algebras are discussed.
\end{abstract}
\medskip
\section{Introduction}

Poisson algebras appear naturally in Hamiltonian mechanics, and play a central role in the study
of Poisson geometry and quantum groups. With the development of Poisson algebras in the past decades,
many important generalizations have been obtained in both commutative and noncommutative settings:
Poisson PI algebras \cite{MPR}, graded Poisson algebras \cite{CFL}, double Poisson algebras \cite{V}, Quiver Poisson algebras \cite{YYZ}, noncommutative Leibniz-Poisson algebras \cite{CD},
Left-right noncommutative Poisson algebras \cite{CDL} and differential graded Poisson algebras \cite{LWZ}, etc. One of most interesting features in this area is the Poisson universal enveloping algebra, which was first introduced by Oh \cite{O} in order to describe the category of Poisson modules. Since then, Poisson universal enveloping algebras have been studied in a series of papers \cite{OPS, Um, YYY}. In particular, the third author and the fourth author of the present paper studied the universal enveloping algebras of Poisson Ore-extensions and DG Poisson algebra \cite{LWZ2, LWZ}.

We know that, the notion of Poisson Hopf algebras originally arises from Poisson geometry and quantum groups. For instance, the coordinate ring of a Poisson algebraic group is a Poisson Hopf algebra \cite{KS}. Recently, Poisson Hopf algebras are studied by many authors from different perspectives \cite{LWZ1, Oh, Oh2, T}. In \cite{Oh}, Oh developed the theory of universal enveloping algebras for Poisson algebras, and then proved that the universal enveloping algebra $A^e$ of a Poisson Hopf algebra $A$ is a Hopf algebra. In addition, the last two authors of this paper studied the most basic properties for universal enveloping algebras of Poisson Hopf algebras \cite{LWZ1}, which would help us to understand Hopf algebras in general.

Motivated by the notion of differential graded Poisson algebra, our aim in this paper is to study Poisson Hopf algebras and their universal enveloping algebras in the DG setting. Roughly speaking, a DG Poisson Hopf algebra is a DG Poisson algebra together with a Hopf structure satisfying certain compatible conditions; see Definition \ref{defn3.1}. As for universal enveloping algebras, there are many equivalent ways to define the universal enveloping algebra of a DG Poisson algebra, among which we choose to use the universal property; see Definition \ref{defn2}.

Note that the universal enveloping algebra $A^e$ of a Poisson bialgebra $A$ is a bialgebra, as an extension, we prove the following result:
\begin{theorem}\label{thm35}
If $(A, u, \eta, \Delta, \varepsilon, \{\cdot, \cdot\}, d)$ is a DG
Poisson bialgebra, then
\begin{center}
$(A^e, u^e, \eta^e, \Delta^e, \varepsilon^e, d^e)$
\end{center}
is a DG bialgebra such that
\begin{align*}
\Delta^em&=(m\otimes m)\Delta   &  \Delta^eh&=(m\otimes h+h\otimes
m)\Delta\\
\varepsilon^em&=\varepsilon  &   \varepsilon^eh&=0.
\end{align*}
\end{theorem}

It should be noted that the universal enveloping algebra $A^e$ of a DG Poisson Hopf algebra $A$ is just a DG bialgebra. But in some special cases,
$A^e$ can be endowed with the Hopf structure, such that $A^e$ becomes a DG Hopf algebra:
\begin{theorem}\label{coro36}
Let $(A, u, \eta, \Delta, \varepsilon, S, \{\cdot, \cdot\}, d)$ be a
DG Poisson Hopf algebra. Suppose that $A^{eop}$ is the DG opposite algebra of $A^e$. Then:
 \begin{enumerate}
   \item There exist a DG algebra homomorphism $S^e:A^e\rightarrow A^{eop}$ such that
         \begin{center}
         $S^em=mS$,\quad \quad $S^eh=hS$.
         \end{center}
   \item  $\{S(a_{(1)}), a_{(2)}\}=0$, where $\Delta(a)=a_{(1)}\otimes a_{(2)}$ for all $a\in A$, if and only if
          \begin{center}
          $(A^e, u^e, \eta^e, \Delta^e, \varepsilon^e, S^e, d^e)$
          \end{center}
          is a DG Hopf algebra.
 \end{enumerate}
\end{theorem}

The paper is organized as follows. In Section 2, we briefly review
some basic concepts and results related to DG algebras, DG Poisson
algebras, DG Hopf algebras and universal enveloping algebras of DG
Poisson algebras. Section 3 is devoted to the study of some
definitions and properties of DG Poisson Hopf algebras and universal
enveloping algebras. In particular, we show that there is a natural
DG bialgebra structure on the universal enveloping algebra $A^e$ for
a DG Poisson bialgebra $A$.

Throughout the whole paper, $\mathbb{Z}$ denotes the set of
integers, $\Bbbk$ denotes a base field of characteristic zero unless otherwise stated, all (graded)
algebras are assumed to have an identity and all (graded) modules
are assumed to be unitary. We always take the grading to be $\mathbb{Z}$-graded.

\medskip
\section{Preliminaries}
In this section, we will recall some definitions and results of DG
Poisson algebras, universal enveloping algebras and DG Hopf
algebras.

\subsection{DG Poisson algebras} By a graded algebra $A$ we mean a
$\mathbb{Z}$-graded algebra $(A, u, \eta)$, where $u: A\otimes
A\rightarrow A$ and $\eta: \Bbbk\rightarrow A$ are called the
multiplication and unit of $A$, respectively. For convenience, we
shall write $u(a\otimes b)$ as $ab$, $\forall a, b\in A$, whenever
this does not cause confusion. A DG algebra is a graded algebra with
a $\Bbbk$-linear homogeneous map $d: A\rightarrow A$ of degree 1,
which is also a graded derivation. Let $A$, $B$ be two DG algebras
and $f: A\rightarrow B$ be a graded algebra map of degree zero. Then
$f$ is called a DG algebra map if $f$ commutes with the
differentials.

\begin{defn}
Let $(A, \cdot)$ be a graded $\Bbbk$-algebra. If there is a
$\Bbbk$-linear map
\begin{center}
$\{\cdot, \cdot\}: A\otimes A\rightarrow A$
\end{center}
of degree $p$ such that

\begin{enumerate}
\item[(\rmnum{1})] $(A, \{\cdot, \cdot\})$ is a graded Lie algebra. That is to
say, we have
\begin{enumerate}
\item[(\rmnum{1a})] $\{a, b\}=-(-1)^{(|a|+p)(|b|+p)}\{b, a\}$;

\item[(\rmnum{1b})] $\{a, \{b, c\}\}=\{\{a, b\}, c\}+(-1)^{(|a|+p)(|b|+p)}\{b, \{a, c\}\}$,
\end{enumerate}

\item[(\rmnum{2})] (graded commutativity): $a\cdot b=(-1)^{|a||b|}b\cdot a$;

\item[(\rmnum{3})] (biderivation property): $\{a, b\cdot c\}=\{a, b\}\cdot c+(-1)^{(|a|+p)|b|}b\cdot \{a, c\}$,
\end{enumerate}
for any homogeneous elements $a, b, c\in A$, then $A$ is called a
graded Poisson algebra \cite{CFL}. If in addition, there is a
$\Bbbk$-linear homogeneous map $d: A\rightarrow A$ of degree 1 such
that $d^2=0$ and

\begin{enumerate}
\item[(\rmnum{4})] $d(\{a, b\})=\{d(a), b\}+(-1)^{(|a|+p)}\{a, d(b)\}$;

\item[(\rmnum{5})] $d(a\cdot b)=d(a)\cdot b+(-1)^{|a|}a\cdot d(b)$,
\end{enumerate}
for any homogeneous elements $a, b\in A$, then $A$ is called a DG
Poisson algebra, which is usually denoted by $(A, \cdot, \{\cdot,
\cdot\}, d)$, or simply by $(A, \{\cdot, \cdot\}, d)$ or $A$ if no
confusions arise.
\end{defn}

\begin{lemma}\label{lem22}\cite{LWZ}
Let $(A, \cdot, \{\cdot, \cdot\}_A, d_A)$ and $(B, *, \{\cdot,
\cdot\}_B, d_B)$ be any two DG Poisson algebras with Poisson brackets of degree $p$. Then $(A\otimes B, \bigstar, \{\cdot, \cdot\}, d)$ is a DG Poisson algebra, where
\begin{align*}
(a\otimes b)\bigstar(a'\otimes b'):=(-1)^{|a'||b|}(a\cdot a')\otimes (b*b')&,\\
d(a\otimes b):=d_A(a)\otimes b+(-1)^{|a|}a\otimes d_B(b)&,\\
\{a\otimes b, a'\otimes b'\}:=(-1)^{(|a'|+p)|b|}\{a, a'\}_A\otimes
(b*b')+(-1)^{(|b|+p)|a'|}(a\cdot a')\otimes \{b, b'\}_B&,
\end{align*}
for any homogeneous elements $a, a'\in A$ and $b, b'\in B$.
\end{lemma}

For DG Poisson algebras $A$ and $B$, a graded algebra homomorphism
$\phi:A\rightarrow B$ is said to be a DG Poisson algebra
homomorphism if $\phi\circ d_A=d_B\circ \phi$ and $\phi(\{a, b\}_A)=\{\phi(a), \phi(b)\}_B$ for all
homogeneous elements $a, b\in A$. We denote by \textbf{DG(P)A} the category of DG (Poisson) algebras
whose morphism space consists of DG (Poisson) algebras homomorphism.

For a DG algebra $B$, assume throughout the paper that, $B_P$ will
be the DG Poisson algebra $B$ with the standard graded Lie bracket
$[a, b]=ab-(-1)^{(|a|+p)(|b|+p)}ba$ for all homogeneous elements $a, b\in
B$, where $p$ is the degree of the Poisson bracket for $B$.

Let us review the definition of a universal enveloping algebra of
$A\in \textbf{DGPA}$.
\begin{defn}\label{defn2}
For a DG Poisson algebra A, a quadruple $(A^{e}, \alpha, \beta,
\partial)$, where $A^e\in \textbf{DGA}$, $\alpha: (A, d)\rightarrow
(A^{e}, \partial)$ is a DG algebra map and $\beta: (A, \{\cdot,
\cdot\}, d)\rightarrow (A^{e}_P, [\cdot, \cdot], \partial)$ is a DG
Lie algebra map such that
\begin{align*}
\alpha(\{a, b\})&=\beta(a)\alpha(b)-(-1)^{(|a|+p)|b|}\alpha(b)\beta(a),\\
\beta(ab)&=\alpha(a)\beta(b)+(-1)^{|a||b|}\alpha(b)\beta(a),
\end{align*}
for any homogeneous elements $a, b\in A$ and $p$ is the degree of the Poisson bracket for $A$, is called the universal
enveloping algebra of $A$ if $(A^{e}, \alpha, \beta, \partial)$
satisfies the following: if $(D, \delta)$ is a DG algebra, $f$ is a
DG algebra map from $(A, d)$ into $(D, \delta)$ and $g$ is a DG Lie
algebra map from $(A, \{\cdot, \cdot\}, d)$ into $(D_P, [\cdot,
\cdot], \delta)$ such that
\begin{align*}
f(\{a, b\})&=g(a)f(b)-(-1)^{(|a|+p)|b|}f(b)g(a),\\
g(ab)&=f(a)g(b)+(-1)^{|a||b|}f(b)g(a),
\end{align*}
for all $a, b\in A$, then there exists a unique DG algebra map
$\phi: (A^{e}, \partial)\rightarrow (D, \delta)$, such that
$\phi\alpha=f$ and $\phi\beta=g$.
\end{defn}

For every DG Poisson algebra $A$, there exists a unique universal
enveloping algebra $A^e$ up to isomorphic, that $A^e$ is generated
by $m(A)$ and $h(A)$ by \cite{LWZ} and that $h(1)=0$. Note that
$A^e$ is a DG algebra, such that a $\Bbbk$-vector space $M$ is a DG
Poisson $A$-module if and only if $M$ is a DG $A^e$-module.

As an application of the ``universal property" of universal
enveloping algebras, we have
\begin{lemma}\label{lem24} \cite{LWZ}
For any $A, B\in \textbf{DGPA}$, suppose that $(A^e, m_A, h_A)$ and $(B^e, m_B, h_B)$ are the universal enveloping algebra of $A$ and $B$, respectively. Then $(A^e\otimes B^e, m_A\otimes m_B, m_A\otimes
h_B+(-1)^{p|B|}h_A\otimes m_B)$ is the universal enveloping algebra for $A\otimes B$, where $p$ is the degree of the Poisson bracket for both $A$ and $B$. We use $(-1)^{p|B|}h_A\otimes m_B$ to mean that
\begin{center}
$((-1)^{p|B|}h_A\otimes m_B)(a\otimes b)=(-1)^{p|b|}h_A(a)\otimes m_B(b)$
\end{center}
for any $a\in A, b\in B$.
\end{lemma}

\begin{lemma}\label{lem25}
Let $(A^e, m_A, h_A)$ and $(B^e, m_B, h_B)$ be universal enveloping
algebras for DG Poisson algebras $A$ and $B$ respectively. Assume that $p$ is the degree of the Poisson bracket for both $A$ and $B$. If $\phi: A\rightarrow B$ is a DG Poisson homomorphism, then there
exists a unique DG algebra homomorphism $\phi^e: A^e\rightarrow B^e$
such that $\phi^em_A=m_B\phi$ and $\phi^eh_A=h_B\phi$.
\begin{align*}
\xymatrix{
  A \ar[d]_{m_A, ~h_A} \ar[rr]^{\phi}
             &   & B \ar[d]^{m_B, ~h_B}  \\
  A^{e} \ar[rr]_{\phi^e}
             &   & B^e             }
\end{align*}
\end{lemma}

\begin{proof}
Note that $m_B$ is a DG algebra map and $h_B$ is a DG Lie algebra
map such that
\begin{align*}
m_B(\{b, b'\})&=h_B(b)m_B(b')-(-1)^{(|b|+p)|b'|}m_B(b')h_B(b),\\
h_B(bb')&=m_B(b)h_B(b')+(-1)^{|b||b'|}m_B(b')h_B(b).
\end{align*}
We have
\begin{align*}
m_B\phi(aa')&=m_B(\phi(a)\phi(a'))=m_B\phi(a)m_B\phi(a'),\\
h_B\phi(\{a, a'\})&=h_B(\{\phi(a), \phi(a')\})=[h_B\phi(a),
h_B\phi(a')],
\end{align*}
\begin{center}
$m_B\phi(\{a, a'\})=m_B(\{\phi(a),
\phi(a')\})=h_B\phi(a)m_B\phi(a')-(-1)^{(|a|+p)|a'|}m_B\phi(a')h_B\phi(a)$
\end{center}
and
\begin{center}
$h_B\phi(aa')=h_B(\phi(a)\phi(a'))=m_B\phi(a)h_B\phi(a')+(-1)^{|a||a'|}m_B\phi(a')h_B\phi(a)$,
\end{center}
since $\phi: A\rightarrow B$ is a DG Poisson homomorphism. Further,
it is easy to see
\begin{center}
$m_B\phi(d_A(a))=m_Bd_B(\phi(a))=d_{B^e}m_B\phi(a)$
\end{center}
and
\begin{center}
$h_B\phi(d_A(a))=h_Bd_B(\phi(a))=d_{B^e}h_B\phi(a)$.
\end{center}
Thus there exists a unique  DG algebra homomorphism $\phi^e$ from
$A^e$ into $B^e$ such that $\phi^em_A=m_B\phi$ and
$\phi^eh_A=h_B\phi$.
\end{proof}

\subsection{DG Hopf algebras}
In this subsection, we recall some definitions and properties of
graded Hopf algebras and DG Hopf algebras.

In the remainder of the paper, the twisting map $T: V\otimes
W\rightarrow W\otimes V$ will frequently appear. It is defined for
homogeneous elements $v\in V$ and $w\in W$ by
\begin{center}
$T(v\otimes w)=(-1)^{|v||w|}w\otimes v$
\end{center}
and extends to all elements of $V$ and $W$ through linearity.

A graded coalgebra $C$ over $\Bbbk$ is a $\mathbb{Z}$-graded vector
space with the graded vector space homomorphisms of degree 0 $\Delta: C\rightarrow C\otimes C$ and $\varepsilon: C\rightarrow
\Bbbk$ such that the following diagrams
\begin{eqnarray*}
 \xymatrix{
  C \ar[dd]_{\Delta} \ar[rr]^{\Delta}  & & C\otimes C \ar[dd]^{I\otimes \Delta}  &
   & C\otimes C   \ar[dl]_{I\otimes ~\varepsilon}  \ar[dr]^{\varepsilon\otimes ~I}  \\
     & &  & C\otimes \Bbbk
   &   & \Bbbk\otimes C  \\
  C\otimes C \ar[rr]^{\Delta\otimes I} &  & C\otimes C\otimes C  &   &
   C \ar[uu]_{\Delta}  \ar[ul]^{\cong} \ar[ur]_{\cong}}
\end{eqnarray*}
are commutative, where $I: C\rightarrow C$ is the identity homomorphism, $\Delta$ and
$\varepsilon$ are called the comultiplication and counit of $C$,
respectively. Note that the commutativity of the first diagram is
equivalent to $(\Delta\otimes I)\Delta=(I\otimes \Delta)\Delta$,
called the coassociativity of $\Delta$. Further, $C$ is graded cocommutative if ~$\Delta=T\Delta$, where $T: C\otimes C\rightarrow C\otimes C$ is the twisting morphism.

For any $c\in C$, we write $\Delta(c)=c_{(1)}\otimes c_{(2)}$, which
is basically the Sweedler's notation with the summation sign
$\Sigma$ omitted. In this notation, the comultiplication and counit
property may be expressed as
\begin{align*}
(\Delta\otimes I)\Delta(c)&=(I\otimes
\Delta)\Delta(c)=c_{(1)}\otimes c_{(2)}\otimes c_{(3)},\\
c&=\varepsilon(c_{(1)})c_{(2)}=c_{(1)}\varepsilon(c_{(2)}),
\end{align*}
for any homogeneous elements $c\in C$, respectively.

A homomorphism of graded coalgebras $f: A\rightarrow B$ is a
homomorphism of graded vector spaces such that the diagrams
\begin{align*}
 \xymatrix{
  A \ar[dd]_{f} \ar[rr]^{\Delta_A}  & & A\otimes A \ar[dd]^{f\otimes f}  &
   & & & A  \ar[dl]_{\varepsilon_A}   \ar[dd]_{f}  \\
     & & & & &\Bbbk
   &   &  \\
  B \ar[rr]^{\Delta_B} &  & B\otimes B  &
  &  & & B  \ar[ul]^{\varepsilon_B}   }
\end{align*}
are commutative.

Let $H$ be a graded algebra with multiplication $u$ and unit $\eta$,
and at the same time a graded coalgebra with comultiplication
$\Delta$ and counit $\varepsilon$. Then $H$ is called a graded
bialgebra provided that one of the following equivalent conditions
are satisfied:
\begin{enumerate}
  \item[(\rmnum{1})] $\Delta$ and $\varepsilon$ are graded algebra homomorphisms;
  \item[(\rmnum{2})] $u$ and $\eta$ are graded coalgebra homomorphisms.
\end{enumerate}
Further, if $H$ admits a graded vector space homomorphism
$S: H\rightarrow H$ of degree 0, which satisfies the following
defining relation
\begin{center}
$u(I\otimes S)\Delta=u(S\otimes I)\Delta=\eta\varepsilon$,
\end{center}
then $H$ is called a graded Hopf algebra, and $S$ is called the
antipode of $H$.

The antipode $S$ has the following properties \cite{GZB, MM, S}.
\begin{lemma}\label{lem26}
Let $H$ be a graded Hopf algebra and $S$ its antipode; then
\begin{enumerate}
  \item $S\cdot u=u\cdot T\cdot (S\otimes S)$,
  \item $S\cdot \eta=\eta$,
  \item $\varepsilon\cdot S=\varepsilon$,
  \item $T\cdot (S\otimes S)\cdot \Delta=\Delta\cdot S$,
  \item if $H$ is graded commutative or graded cocommutative, then $S\cdot S=I$,
\end{enumerate}
where $I: H\rightarrow H$ is the identity morphism and $T: H\otimes
H\rightarrow H\otimes H$ is the twisting morphism.
\end{lemma}

\begin{defn} \cite{U}
Let $H$ be a graded $\Bbbk$-vector space. If there is a
$\Bbbk$-linear homogeneous map $d: H\rightarrow H$ of degree 1 such
that $d^2=0$ and
\begin{enumerate}
\item[(\rmnum{1})] $(H, u, \eta, d)$ is a DG algebra. That
is to say, we have
\begin{enumerate}
\item[(\rmnum{1a})] $(H, u, \eta)$ is a graded algebra;

\item[(\rmnum{1b})] $d$ is a (algebra) derivation of degree 1, that
means,
\begin{center}
$d(ab)=d(a)b+(-1)^{|a|}ad(b)$,
\end{center}
for any homogeneous elements
$a, b\in H$.
\end{enumerate}

\item[(\rmnum{2})] $(H, \Delta, \varepsilon, d)$ is a DG coalgebra. That
is, we have
\begin{enumerate}
\item[(\rmnum{2a})] $(H, \Delta, \varepsilon)$ is a graded
coalgebra;

\item[(\rmnum{2b})] $d$ is a coderivation of degree 1, that
means, $\varepsilon d=0$ and
\begin{center}
$\Delta d=(d\otimes I+T(d\otimes I)T)\Delta.$
\end{center}
\end{enumerate}

\item[(\rmnum{3})] $(H, u, \eta, \Delta, \varepsilon)$ is a graded
 bialgebra.
\end{enumerate}
Then $H$ is called a DG bialgebra. Moreover, if $(H, u, \eta, \Delta, \varepsilon, S)$ is a
graded Hopf algebra, then $H$ is called a DG Hopf algebra, which is
usually denoted by $(H, u, \eta, \Delta, \varepsilon, S, d)$.
\end{defn}

\medskip
\section{DG Poisson Hopf algebras and universal enveloping algebras}
Now we will introduce the notion of DG Poisson Hopf algebras and
prove that the universal enveloping algebra $A^e$ of a DG Poisson
bialgebra $A$ is a DG bialgebra.

\begin{defn} \label{defn3.1}
Let $A$ be a graded $\Bbbk$-vector space. If there is a
$\Bbbk$-linear map
\begin{center}
$\{\cdot, \cdot\}: A\otimes A\rightarrow A$
\end{center}
of degree $p$ such that:

\begin{enumerate}
\item[(\rmnum{1})] $(A, u, \eta, \{\cdot, \cdot\})$ is a graded Poisson algebra;

\item[(\rmnum{2})] $(A, u, \eta, \Delta, \varepsilon)$ is a graded bialgebra;

\item[(\rmnum{3})] $\Delta(\{a, b\}_A)=\{\Delta(a), \Delta(b)\}_{A\otimes A}$ for all $a, b\in A$,
where the Poisson bracket $\{\cdot, \cdot\}_{A\otimes A}$ on
$A\otimes A$ is defined by
\begin{center}
$\{a\otimes a', b\otimes b'\}_{A\otimes A}=(-1)^{(|b|+p)|a'|}\{a,
b\}\otimes a'b'+(-1)^{(|a'|+p)|b|}ab\otimes \{a', b'\}$
\end{center}
for any homogeneous elements $a, b, a', b'\in A$.
\end{enumerate}
Then $A$ is called a graded Poisson bialgebra. If in addition, there
is a $\Bbbk$-linear homogeneous map $d: A\rightarrow A$ of degree 1
such that $d^2=0$ and

\begin{enumerate}
\item[(\rmnum{4a})] $d(\{a, b\})=\{d(a), b\}+(-1)^{(|a|+p)}\{a, d(b)\}$;

\item[(\rmnum{4b})] $d(a\cdot b)=d(a)\cdot b+(-1)^{|a|}a\cdot d(b)$;

\item[(\rmnum{4c})] $\varepsilon d=0$ and $\Delta d(a)=d(a_{(1)})\otimes a_{(2)}
+(-1)^{|a_{(1)}|}a_{(1)}\otimes d(a_{(2)})$, where
$\Delta(a)=a_{(1)}\otimes a_{(2)}$,
\end{enumerate}
for any homogeneous elements $a, b\in A$, then $(A, u, \eta, \Delta,
\varepsilon, \{\cdot, \cdot\}, d)$ is called a DG Poisson bialgebra.
Further, if $A$ admits a graded vector space homomorphism
$S: A\rightarrow A$ of degree 0 which satisfies the following
defining relation
\begin{center}
$u(I\otimes S)\Delta=u(S\otimes I)\Delta=\eta\varepsilon$.
\end{center}
Then $A$ is called a DG Poisson Hopf algebra,  which is usually
denoted by
\begin{center}
$(A, u, \eta, \Delta, \varepsilon, S, \{\cdot, \cdot\},
d)$.
\end{center}
\end{defn}

From now on, we would like to use the degree of the Poisson bracket
is zero in our definition of DG Poisson Hopf algebra, but the result
obtained in this paper are also true for DG Poisson Hopf algebra of
degree $p$ with some expected signs, where $p\in \mathbb{Z}$ is the
degree of the Poisson bracket.

As a generalization of Lemma \ref{lem22}, we have
\begin{prop}
Let $(A, u_A, \eta_A, \Delta_A, \varepsilon_A, S_A, \{\cdot, \cdot\}_A,
d_A)$ and $(B, u_B, \eta_B, \Delta_B, \varepsilon_B, S_B, \{\cdot, \cdot\}_B,
d_B)$ be any two DG Poisson Hopf algebras with Poisson brackets of degree $0$. Then \[(A\otimes B, u, \eta, \Delta, \varepsilon, S, \{\cdot, \cdot\}, d)\] is a DG Poisson Hopf algebra, where
\begin{align*}
&u((a\otimes b)\otimes (a'\otimes b')):=(-1)^{|a'||b|}aa'\otimes bb', \quad \ \ \ \quad \quad \ \eta(1_{\Bbbk}):=1_A\otimes 1_B ,\\
&\Delta(a\otimes b):=(-1)^{|a_{(2)}||b_{(1)}|}a_{(1)}\otimes b_{(1)}\otimes a_{(2)}\otimes b_{(2)}, \ \ \ \quad \quad \epsilon(a\otimes b):=\epsilon_{A}(a)\epsilon_{B}(b),\\
&d(a\otimes b):=d_A(a)\otimes b+(-1)^{|a|}a\otimes d_B(b), \quad \quad \quad \quad \quad S(a\otimes b):=S_{A}(a)\otimes S_{B}(b),\\
&\{a\otimes b, a'\otimes b'\}:=(-1)^{|a'||b|}(\{a, a'\}_A\otimes
bb'+aa'\otimes \{b, b'\}_B),
\end{align*}
for any homogeneous elements $a, a'\in A$ and $b, b'\in B$. As usual, we write $\Delta(a)=a_{(1)}\otimes a_{(2)}$ and $\Delta(b)=b_{(1)}\otimes b_{(2)}$.
\end{prop}

\begin{proof}
It is verified routinely that $(A\otimes B, u, \eta, \Delta, \varepsilon, S, \{\cdot, \cdot\}, d)$ is a DG Poisson Hopf algebra.
\end{proof}

Now we give some examples of DG Poisson Hopf algebras.

\begin{exa}  \label{exa30}
Let $(A, u, \eta, \Delta, \varepsilon, S, \{\cdot, \cdot\}, d)$ be any DG Poisson Hopf algebra with Poisson brackets of degree $0$. Then \[(A^{op}, u^{op}, \eta, \Delta^{op}, \varepsilon, S, \{\cdot, \cdot\}^{op}, d)\] is also a DG Poisson Hopf algebra, where
\begin{align*}
u^{op}(a\otimes b)=(-1)^{|a||b|}b\cdot a=a\cdot b=u(a\otimes b),&\\
\Delta^{op}=T\Delta,&\\
\{a, b\}^{op}=(-1)^{|a||b|}\{b, a\}=-\{a, b\},&
\end{align*}
for any homogeneous elements $a, b\in A$, and $T: A\otimes A\rightarrow A\otimes A$ is the twisting
morphism.
\end{exa}

From the constructions of opposite algebra and tensor product of DG Poisson Hopf algebras, we
have the following observation:
\begin{prop}
Let $A$ and $B$ be any two DG Poisson Hopf algebras. Then \[(A\otimes B)^{op}=A^{op}\otimes B^{op}.\]
\end{prop}

In \cite{LWZ1}, there are a lot of examples about Poisson Hopf algebras, which can be considered as DG Poisson Hopf algebras concentrated in degree 0 with trivial differential. For instance, we have the following example.
\begin{exa} \label{exa31}
Assume that the base field $\Bbbk$ has characteristic $p>2$. Let
\begin{center}
$B=\Bbbk[x, y, z]/(x^p, y^p, z^p)$
\end{center}
be the restricted symmetric algebra with three variables. Then $B$ becomes a Poisson Hopf algebra via
\begin{align*}
&\Delta(x)=x\otimes 1+1\otimes x,\\
&\Delta(y)=y\otimes 1+1\otimes y,\\
&\Delta(z)=z\otimes 1+1\otimes z-2x\otimes y,\\
&S(x)=-x, \quad S(y)=-y, \quad S(z)=-z-2xy,\\
&\{x, y\}=y, \quad \{y, z\}=y^2, \quad \{x, z\}=z.
\end{align*}
\end{exa}
In the above example, one should view it as a Poisson version of the Hopf algebra in \cite{NWW}. Similarly, we can obtain other Poisson Hopf algebras from connected Hopf algebras in \cite{NWW}.

\begin{exa} \label{exa32}
Let $(L, d_L)$ be a finite dimensional DG Lie algebra over $\Bbbk$
with standard graded Lie bracket $[\cdot, \cdot]$ and let $(S(L), u,
\eta)$ be the graded symmetric algebra of $L$, i.e.,

$$S(L):=\frac{T(L)}{(a\otimes b-(-1)^{|a||b|}b\otimes a)},$$

for any $a, b\in L$, where $u$ is a multiplication and $\eta$ is a
unit. The differential $d_L$ of $L$ can be extended to the graded
symmetric algebra $S(L)$ such that $S(L)$ becomes a graded
commutative DG algebra and denote by $d$ the differential in $S(L)$.
Here, we consider the total grading on $S(L)$ coming from the
grading of $L$. Moreover, the Lie bracket on $L$ also can be
extended to a Poisson bracket on $S(L)$ by graded Jacobi identity
such that
\begin{center}
$\{a, b\}_{S(L)}:=[a, b],$
\end{center}
for any $a, b\in L$. Hence, the graded symmetric algebra $S(L)$ over
$L$ has a natural DG Poisson algebra structure.

Now, we introduce the graded Hopf structure for $S(L)$.  Denote by $\Delta$
the familiar comultiplication $\Delta: S(L)\rightarrow S(L)\otimes
S(L)$ such that
\begin{align*}
\Delta(a)&=a\otimes 1+1\otimes a, ~~~~~~~~~~\ \forall a\in L,\\
\Delta(uv)&=\Delta(u)\Delta(v), ~~~~~~~~\qquad \forall u, v\in S(L)
\end{align*}
and let the morphism $\varepsilon: S(L)\rightarrow \Bbbk$ be defined
by
\begin{align*}
\varepsilon(1)&=1,\\
\varepsilon(a)&=0, ~~~~~\qquad \qquad \ \ \ \ \forall a\in L,\\
\varepsilon(uv)&=\varepsilon(u)\varepsilon(v), ~~~~~\qquad \ \forall
u, v\in S(L),
\end{align*}
then $S(L)$ constitutes a $\mathbb{Z}$-graded bialgebra with
comultiplication $\Delta$ and counit $\varepsilon$. Fix a
$\Bbbk$-basis $x_1, x_2,..., x_n$ of $L$. Note that $S(L)$ is the
graded commutative polynomial ring $k[x_1, x_2,..., x_n]$, we can
prove the formula $\Delta(\{a, b\})=\{\Delta(a), \Delta(b)\}$ by
using the induction on degree of homogeneous elements of $S(L)$. In
fact, we have
\begin{align*}
&\Delta(\{x_i, x_j\}_{S(L)})\\=&\Delta(x_ix_j-(-1)^{|x_i||x_j|}x_jx_i)\\
=&\Delta(x_i)\Delta(x_j)-(-1)^{|x_i||x_j|}\Delta(x_j)\Delta(x_i)\\
=&(x_i\otimes 1+1\otimes x_i)(x_j\otimes 1+1\otimes x_j)-(-1)^{|x_i||x_j|}(x_j\otimes 1+1\otimes x_j)(x_i\otimes 1+1\otimes x_i)\\
=&x_ix_j\otimes 1+x_i\otimes x_j+(-1)^{|x_i||x_j|}(x_j\otimes x_i)+1\otimes x_ix_j\\&-(-1)^{|x_i||x_j|}(x_jx_i\otimes 1+x_j\otimes x_i+(-1)^{|x_i||x_j|}(x_i\otimes x_j)+1\otimes x_jx_i)\\
=&x_ix_j\otimes 1+1\otimes x_ix_j-(-1)^{|x_i||x_j|}(x_jx_i\otimes
1)-(-1)^{|x_i||x_j|}(1\otimes x_jx_i)
\end{align*}
and
\begin{align*}
&\{\Delta(x_i), \Delta(x_j)\}\\
=&\{x_i\otimes 1+1\otimes x_i, x_j\otimes 1+1\otimes x_j\}\\
=&\{x_i, x_j\}\otimes 1+1\otimes \{x_i, x_j\}\\
=&x_ix_j\otimes 1-(-1)^{|x_i||x_j|}(x_jx_i\otimes 1)+1\otimes
x_ix_j-(-1)^{|x_i||x_j|}(1\otimes x_jx_i),
\end{align*}
and so we have $\Delta(\{x_i, x_j\})=\{\Delta(x_i), \Delta(x_j)\}.$ For any
monomials $a, b, c\in S(L)$, assume that $\Delta(\{a,
c\})=\{\Delta(a), \Delta(c)\}$ and $\Delta(\{b, c\})=\{\Delta(b),
\Delta(c)\}$. Now we have
\begin{align*}
\Delta(\{ab, c\})&=\Delta(a\{b, c\}+(-1)^{|a||b|}b\{a, c\})\\
&=\Delta(a)\Delta(\{b, c\})+(-1)^{|a||b|}\Delta(b)\Delta(\{a, c\})\\
&=\Delta(a)\{\Delta(b), \Delta(c)\}+(-1)^{|a||b|}\Delta(b)\{\Delta(a), \Delta(c)\}\\
&=\{\Delta(ab), \Delta(c)\}.
\end{align*}
Applying the induction and biderivation property, we have
\begin{center}
$\Delta(\{a, b\})=\{\Delta(a), \Delta(b)\}$,
\end{center}
for all $a, b\in S(L)$. In fact, it is also verified easily using
the induction on degree of homogeneous elements of $S(L)$ that
$\Delta d(a)=d(a_{(1)})\otimes a_{(2)} +(-1)^{|a_{(1)}|}a_{(1)}\otimes
d(a_{(2)})$, $\varepsilon d(a)=0$ and $\varepsilon(\{a, b\})=0$ for all $a, b\in S(L)$, where $\Delta(a)=a_{(1)}\otimes
a_{(2)}$. Hence $(S(L), u, \eta, \Delta, \varepsilon, \{\cdot,
\cdot\}, d)$ is a DG Poisson bialgebra. In order to finish the
proof, it suffice to prove that
\begin{equation*}
u(S\otimes I)\Delta=u(I\otimes S)\Delta=\eta\varepsilon,
\end{equation*}
where the $\Bbbk$-linear homogeneous map $S: S(L)\rightarrow S(L)$ of
degree 0 is defined by
\begin{align*}
S(1)&=1,\\
S(a)&=-a, ~~~~\qquad\forall a\in L,\\
S(uv)&=(-1)^{|u||v|}S(v)S(u)
\end{align*}
for homogeneous elements of $S(L)$ which extends to all $S(L)$ by
linearity. Then we have
\begin{center}
$u(S\otimes I)\Delta(x_i)=u(S\otimes I)(x_i\otimes 1+1\otimes
x_i)=S(x_i)+x_i=-x_i+x_i=0=\eta\varepsilon(x_i)$
\end{center}
and
\begin{center}
$u(I\otimes S)\Delta(x_i)=u(I\otimes S)(x_i\otimes 1+1\otimes
x_i)=x_i+S(x_i)=x_i-x_i=0=\eta\varepsilon(x_i)$.
\end{center}
Now it suffices to prove the formula is true for $ab$ provided that
it is true for any homogeneous elements $a, b\in S(L)$. Thus, assume
that we have the following two equations:
\begin{center}
$u(S\otimes I)\Delta(a)=\eta\varepsilon(a)~~$ and $~~u(S\otimes
I)\Delta(b)=\eta\varepsilon(b)$,
\end{center}
for any homogeneous elements $a, b\in S(L)$. Then we have
\begin{align*}
&u(S\otimes I)\Delta(ab)\\=&u(S\otimes I)(\Delta(a)\Delta(b))
=u(S\otimes I)((-1)^{|a_{(2)}||b_{(1)}|}a_{(1)}b_{(1)}\otimes a_{(2)}b_{(2)})\\
=&(-1)^{|a_{(2)}||b_{(1)}|}S(a_{(1)}b_{(1)})a_{(2)}b_{(2)}=(-1)^{|a_{(2)}||b_{(1)}|+|a_{(1)}||b_{(1)}|}S(b_{(1)})S(a_{(1)})a_{(2)}b_{(2)}\\
=&S(a_{(1)})a_{(2)}S(b_{(1)})b_{(2)}=\eta\varepsilon(a)\eta\varepsilon(b)=\eta\varepsilon(ab).
\end{align*}
Similarly, we can prove the formula $u(I\otimes
S)\Delta=\eta\varepsilon$ is true, which complete the proof.
\end{exa}

\begin{lemma}\label{lem323}
Retain the notations of Example \ref{exa32}, we have that
\begin{equation}
\{S(a_{(1)}), a_{(2)}\}=0,
\end{equation}
where $\Delta(a)=a_{(1)}\otimes a_{(2)}$ for all $a\in S(L)$.
\end{lemma}

\begin{proof}
We proceed the proof using induction on the length of $a\in S(L)$.
Observe that formula (1) is true on any $a\in L$ since
\begin{center}
$\{S(a_{(1)}), a_{(2)}\}=\{S(a), 1\}+\{S(1), a\}=\{-a, 1\}+\{1,
a\}=0,$
\end{center}
where $a\in L$ and $\Delta(a)=a\otimes 1+1\otimes a$.

In order to prove that formula (1) is true for all $a\in S(L)$. It
suffices to prove the formula (1) is true for $ab$ provided that
it is true for any elements $a, b\in S(L)$. Thus, assume that we
have the following two equations:
\begin{center}
$\{S(a_{(1)}), a_{(2)}\}=0, \ \ \ \ \ \ \{S(b_{(1)}), b_{(2)}\}=0,$
\end{center}
for any elements $a, b\in S(L)$. Note that
\begin{center}
$\Delta(ab)=\Delta(a)\Delta(b)=(a_{(1)}\otimes
a_{(2)})(b_{(1)}\otimes b_{(2)})
=(-1)^{|a_{(2)}||b_{(1)}|}a_{(1)}b_{(1)}\otimes a_{(2)}b_{(2)}.$
\end{center}
We have
\begin{align*}
&\{S((ab)_{(1)}), (ab)_{(2)}\}\\=&(-1)^{|a_{(2)}||b_{(1)}|}\{S(a_{(1)}b_{(1)}), a_{(2)}b_{(2)}\}\\
=&(-1)^{|a_{(2)}||b_{(1)}|}\{(-1)^{|a_{(1)}||b_{(1)}|}S(b_{(1)})S(a_{(1)}), a_{(2)}b_{(2)}\}\\
=&(-1)^{|a_{(1)}a_{(2)}||b_{(1)}|}[\{S(b_{(1)})S(a_{(1)}), a_{(2)}\}b_{(2)}+(-1)^{|a_{(1)}b_{(1)}||a_{(2)}|}a_{(2)}\{S(b_{(1)})S(a_{(1)}), b_{(2)}\}]\\
=&(-1)^{|a_{(1)}a_{(2)}||b_{(1)}|}[S(b_{(1)})\{S(a_{(1)}),
a_{(2)}\}b_{(2)}
+(-1)^{|a_{(1)}||a_{(2)}|}\{S(b_{(1)}), a_{(2)}\}S(a_{(1)})b_{(2)}]\\&+(-1)^{|a_{(1)}||a_{(2)}b_{(1)}|}[a_{(2)}S(b_{(1)})\{S(a_{(1)}), b_{(2)}\}+(-1)^{|a_{(1)}||b_{(2)}|}a_{(2)}\{S(b_{(1)}), b_{(2)}\}S(a_{(1)})]\\
=&(-1)^{|a_{(1)}||a_{(2)}b_{(1)}|}[(-1)^{|a_{(2)}||b_{(1)}|}\{S(b_{(1)}), a_{(2)}\}S(a_{(1)})b_{(2)}+a_{(2)}S(b_{(1)})\{S(a_{(1)}), b_{(2)}\}].
\end{align*}
From the structure of DG Poisson Hopf algebra $S(L)$, we have
\[0=\{\varepsilon(a)1_A, b\}=\{S(a_{(1)})a_{(2)},
b\}=S(a_{(1)})\{a_{(2)},
b\}+(-1)^{|b||a_{(2)}|}\{S(a_{(1)}), b\}a_{(2)}\]
and
\[0=\{a, \varepsilon(b)1_A\}=\{a, S(b_{(1)})b_{(2)}\}=\{a,
S(b_{(1)})\}b_{(2)}+(-1)^{|a||b_{(1)}|}S(b_{(1)})\{a, b_{(2)}\},\]
which imply that
\begin{center}
$S(a_{(1)})\{a_{(2)}, b\}=-(-1)^{|b||a_{(2)}|}\{S(a_{(1)}),
b\}a_{(2)}$
\end{center}
and
\begin{center}
$\{a, S(b_{(1)})\}b_{(2)}=-(-1)^{|a||b_{(1)}|}S(b_{(1)})\{a,
b_{(2)}\}$.
\end{center}
Thus
\begin{center}
$(-1)^{|a_{(1)}||a_{(2)}b_{(1)}|}[(-1)^{|a_{(2)}||b_{(1)}|}\{S(b_{(1)}), a_{(2)}\}S(a_{(1)})b_{(2)}+a_{(2)}S(b_{(1)})\{S(a_{(1)}), b_{(2)}\}]=0$
\end{center}
by using the above two formulas. Hence we have the conclusion.
\end{proof}

\begin{lemma}\label{lem33}
If $(A, u, \eta, \Delta, \varepsilon, S, \{\cdot, \cdot\}, d)$ is a
DG Poisson Hopf algebra, then $dS=Sd$, $S(\{a,
b\})=(-1)^{|a||b|}\{S(b), S(a)\}$ and the counit $\varepsilon$ is a
DG Poisson homomorphism.
\end{lemma}

\begin{proof}
Firstly, since $\{a\otimes a', b\otimes b'\}_{A\otimes
A}=(-1)^{|a'||b|}(\{a, b\}\otimes a'b'+ab\otimes \{a', b'\})$ and $\Delta(\{a, b\})=\{\Delta(a), \Delta(b)\}_{A\otimes
A}$, for
all $a, b, a', b'\in A$, we have
\begin{align*}
\{a, b\}&=u(\varepsilon\otimes I)\Delta(\{a, b\})\\
&=u(\varepsilon\otimes I)((-1)^{|b_{(1)}||a_{(2)}|}(\{a_{(1)},
b_{(1)}\}\otimes
a_{(2)}b_{(2)}+a_{(1)}b_{(1)}\otimes \{a_{(2)}, b_{(2)}\}))\\
&=(-1)^{|b_{(1)}||a_{(2)}|}(\varepsilon(\{a_{(1)},
b_{(1)}\})a_{(2)}b_{(2)}+\varepsilon(a_{(1)}b_{(1)})\{a_{(2)}, b_{(2)}\})\\
&=(-1)^{|b_{(1)}||a_{(2)}|}(\varepsilon(\{a_{(1)},
b_{(1)}\})a_{(2)}b_{(2)}+\{a, b\}),
\end{align*}
for all $a, b\in A$, and so we have
\begin{align*}
\varepsilon(\{a,
b\})&=(-1)^{|b_{(1)}||a_{(2)}|}\varepsilon(\varepsilon(\{a_{(1)},
b_{(1)}\})a_{(2)}b_{(2)}+\{a, b\})\\
&=(-1)^{|b_{(1)}||a_{(2)}|}(\varepsilon(\{a_{(1)},
b_{(1)}\})\varepsilon(a_{(2)}b_{(2)})+\varepsilon(\{a, b\}))\\
&=(-1)^{|b_{(1)}||a_{(2)}|}(\varepsilon(\{a, b\})+\varepsilon(\{a,
b\})),
\end{align*}
hence $\varepsilon(\{a, b\})=0$. Note that $\varepsilon d=0$, thus the
counit $\varepsilon$ is a DG Poisson homomorphism.

Secondly, since
\[0=\{\varepsilon(a)1_A, b\}=\{S(a_{(1)})a_{(2)},
b\}=S(a_{(1)})\{a_{(2)},
b\}+(-1)^{|b||a_{(2)}|}\{S(a_{(1)}), b\}a_{(2)}\]
and
\[0=\{a, \varepsilon(b)1_A\}=\{a, S(b_{(1)})b_{(2)}\}=\{a,
S(b_{(1)})\}b_{(2)}+(-1)^{|a||b_{(1)}|}S(b_{(1)})\{a, b_{(2)}\},\]
we have
\begin{center}
$S(a_{(1)})\{a_{(2)}, b\}=-(-1)^{|b||a_{(2)}|}\{S(a_{(1)}),
b\}a_{(2)}$
\end{center}
and
\begin{center}
$\{a, S(b_{(1)})\}b_{(2)}=-(-1)^{|a||b_{(1)}|}S(b_{(1)})\{a,
b_{(2)}\}$.
\end{center}
Note that
\begin{center}
$\Delta(\{a, b\})=\{\Delta(a),
\Delta(b)\}=(-1)^{|b_{(1)}||a_{(2)}|}(\{a_{(1)}, b_{(1)}\}\otimes
a_{(2)}b_{(2)}+a_{(1)}b_{(1)}\otimes \{a_{(2)}, b_{(2)}\})$,
\end{center}
since $A$ is a DG Poisson Hopf algebra. We have
\begin{align*}
0=\varepsilon(\{a, b\})1_A=S(\{a, b\}_{(1)})\{a,
b\}_{(2)}=(-1)^{|b_{(1)}||a_{(2)}|}(S(\{a_{(1)},
b_{(1)}\})a_{(2)}b_{(2)}+S(a_{(1)}b_{(1)})\{a_{(2)}, b_{(2)}\}).
\end{align*}
Thus
\begin{align*}
&S(\{a_{(1)}, b_{(1)}\})a_{(2)}b_{(2)}\\=&-S(a_{(1)}b_{(1)})\{a_{(2)}, b_{(2)}\}=-(-1)^{|a_{(1)}||b_{(1)}|}S(b_{(1)})S(a_{(1)})\{a_{(2)}, b_{(2)}\}\\
=&(-1)^{|a_{(1)}||b_{(1)}|+|b_{(2)}||a_{(2)}|}S(b_{(1)})\{S(a_{(1)}),
b_{(2)}\}a_{(2)}=-(-1)^{|b_{(2)}||a_{(2)}|}\{S(a_{(1)}),
S(b_{(1)})\}b_{(2)}a_{(2)}\\=&-\{S(a_{(1)}), S(b_{(1)})\}a_{(2)}b_{(2)}.
\end{align*}
Therefore
\begin{align*}
&S(\{a, b\})\\=&S(\{a_{(1)}\varepsilon(a_{(2)}),
b_{(1)}\varepsilon(b_{(2)})\})=S(\{a_{(1)},
b_{(1)}\})\varepsilon(a_{(2)})\varepsilon(b_{(2)})\\=&S(\{a_{(1)},
b_{(1)}\})a_{(2)}S(a_{(3)})b_{(2)}S(b_{(3)})=-(-1)^{|b_{(2)}||a_{(3)}|}\{S(a_{(1)}),
S(b_{(1)})\}a_{(2)}b_{(2)}S(a_{(3)})S(b_{(3)})\\
=&-\{S(a_{(1)}), S(b_{(1)})\}a_{(2)}S(a_{(3)})b_{(2)}S(b_{(3)})=-\{S(a_{(1)}), S(b_{(1)})\}\varepsilon(a_{(2)})\varepsilon(b_{(2)})\\
=&-\{S(a_{(1)}\varepsilon(a_{(2)})),
S(b_{(1)}\varepsilon(b_{(2)}))\}=-\{S(a),
S(b)\}=(-1)^{|a||b|}\{S(b), S(a)\}.
\end{align*}
Finally, it only remains to show that $dS=Sd$. Note that
\begin{align*}
dS(a)=&dS(a_{(1)}\varepsilon(a_{(2)}))=d(S(a_{(1)})\varepsilon(a_{(2)}))\\=&d(S(a_{(1)})a_{(2)}S(a_{(3)}))
=dS(a_{(1)})a_{(2)}S(a_{(3)})+(-1)^{|a_{(1)}|}S(a_{(1)})d(a_{(2)}S(a_{(3)}))\\
=&dS(a_{(1)})a_{(2)}S(a_{(3)})+(-1)^{|a_{(1)}|}S(a_{(1)})(d(a_{(2)})S(a_{(3)})+(-1)^{|a_{(2)}|}a_{(2)}dS(a_{(3)}))\\
=&dS(a)+(-1)^{|a_{(1)}|}S(a_{(1)})d(a_{(2)})S(a_{(3)})+(-1)^{|a_{(1)}a_{(2)}|}dS(a),
\end{align*}
we get
\begin{center}
$dS(a)=-(-1)^{|a_{(2)}|}S(a_{(1)})d(a_{(2)})S(a_{(3)}).$
\end{center}
Observe that
\begin{center}
$\Delta d(a)=d(a_{(1)})\otimes
a_{(2)}+(-1)^{|a_{(1)}|}a_{(1)}\otimes d(a_{(2)})$,
\end{center}
we can see
\begin{center}
$0=\varepsilon d(a)1_A=\eta\varepsilon
d(a)=d(a)_{(1)}S(d(a)_{(2)})=d(a_{(1)})S(a_{(2)})+(-1)^{|a_{(1)}|}a_{(1)}Sd(a_{(2)})$,
\end{center}
thus
\begin{center}
$d(a_{(1)})S(a_{(2)})=-(-1)^{|a_{(1)}|}a_{(1)}Sd(a_{(2)})$.
\end{center}
Therefore,
\begin{center}
$dS(a)=-(-1)^{|a_{(2)}|}S(a_{(1)})d(a_{(2)})S(a_{(3)})=S(a_{(1)})a_{(2)}Sd(a_{(3)})=Sd(a)$,
\end{center}
as claimed.
\end{proof}

\begin{lemma}\label{lem324}
If $(A, u, \eta, \Delta, \varepsilon, S, \{\cdot, \cdot\}, d)$ is a
DG Poisson Hopf algebra, then for all $a\in A$, $\{a_{(1)}, S(a_{(2)})\}=0$ if and
only if $\{S(a_{(1)}), a_{(2)}\}=0$, where $\Delta(a)=a_{(1)}\otimes
a_{(2)}$.
\end{lemma}

\begin{proof}
By the definition of a DG Poisson Hopf algebra, $A$ is graded
commutative graded Hopf algebra, then we have $S\cdot S=I$ by Lemma
\ref{lem26}. Note that
\begin{align*}
\{a_{(1)}, S(a_{(2)})\}=\{S\cdot S(a_{(1)}),
S(a_{(2)})\}=(-1)^{|a_{(1)}||a_{(2)}|}S(\{a_{(2)},
S(a_{(1)})\})=-S(\{ S(a_{(1)}), a_{(2)}\})
\end{align*}
by Lemma \ref{lem33}. It is easy to see sufficiency is true.
Similarly, we can prove necessity by the same way. Therefore, we are
done.
\end{proof}

Let us consider the bialgebra structure of universal enveloping algebras
for DG Poisson bialgebras(note: the degree of Poisson bracket is zero). Let $(A, \cdot, \{\cdot, \cdot\}, d)$ be a
DG Poisson algebra. Define a $\Bbbk$-linear map $\{\cdot,
\cdot\}_1$ of degree 0 on $A$ by
\begin{center}
$\{a, b\}_1=(-1)^{|a||b|}\{b, a\}=-\{a, b\}$,
\end{center}
for all $a, b\in A$. Then $A_1=(A, \cdot, \{\cdot, \cdot\}_1, d)$ is
a DG Poisson algebra by Example \ref{exa30}. For a DG algebra $(B, \cdot, d)$, we denote by
$B^{op}=(B, \circ, d)$ the DG opposite algebra of $B$, where $a\circ
b=(-1)^{|a||b|}b\cdot a=a\cdot b$.

\begin{prop}\label{prop34} \cite{LWZ}
Let $(A^e, m, h)$ be the universal enveloping algebra for a DG
Poisson algebra $(A, \cdot, \{\cdot, \cdot\}, d)$. Then $(A^{eop},
m, h)$ is the universal enveloping algebra for $A_1=(A, \cdot,
\{\cdot, \cdot\}_1, d)$.
\end{prop}

Now we are ready to prove Theorem \ref{thm35}.

\begin{proof}[Proof of Theorem 1.1]
Since $\Delta$ is a DG Poisson homomorphism and $(A^e\otimes A^e,
m\otimes m, m\otimes h+h\otimes m)$ is the universal enveloping
algebra of a DG Poisson algebra $A\otimes A$ (see Lemmas \ref{lem22} and \ref{lem24}), there exists a unique DG algebra homomorphism
$\Delta^e: A^e\rightarrow A^e\otimes A^e$ such that
$\Delta^em=(m\otimes m)\Delta$ and $\Delta^eh=(m\otimes h+h\otimes
m)\Delta$ by Lemma \ref{lem25}. Similarly, there exists a DG algebra
homomorphism $\varepsilon^e$ from $A^e$ into $\Bbbk$ such that
$\varepsilon^em=\varepsilon$ and $\varepsilon^eh=0$, since $(\Bbbk,
I_{\Bbbk}, 0)$ is the universal enveloping algebra of the field
$\Bbbk$ with trivial differential and trivial Poisson bracket.

Now we claim that $(A^e, u^e, \eta^e, \Delta^e, \varepsilon^e, d^e)$
is a DG bialgebra. Note that $(A^e, u^e, \eta^e, d^e)$ is a DG
algebra. Next, we show that $(A^e, \Delta^e, \varepsilon^e, d^e)$ is
a DG coalgebra. Since $(A, \Delta, \varepsilon, d)$ is a DG
coalgebra, we have
$a_{(11)}\otimes a_{(12)}\otimes a_{(2)}=a_{(1)}\otimes
a_{(21)}\otimes a_{(22)}$, $\varepsilon(a_{(1)})a_{(2)}=a=a_{(1)}\varepsilon(a_{(2)})$, $\varepsilon d=0$ and $\Delta
d(a)=d(a_{(1)})\otimes a_{(2)}+(-1)^{|a_{(1)}|}a_{(1)}\otimes
d(a_{(2)})$, for all $a\in A$. Thus the coassociativity is immediate from
\begin{align*}
(\Delta^e\otimes I)\Delta^em(a)&=(\Delta^e\otimes I)(m\otimes m)\Delta(a)
=\Delta^em(a_{(1)})\otimes m(a_{(2)})\\&=(m\otimes m)\Delta(a_{(1)})\otimes m(a_{(2)})=m(a_{(11)})\otimes m(a_{(12)})\otimes m(a_{(2)}),
\end{align*}
\begin{align*}
(I\otimes \Delta^e)\Delta^em(a)&=(I\otimes \Delta^e)(m\otimes m)\Delta(a)
=m(a_{(1)})\otimes \Delta^em(a_{(2)})\\&=m(a_{(1)})\otimes (m\otimes m)\Delta(a_{(2)})=m(a_{(1)})\otimes m(a_{(21)})\otimes m(a_{(22)}),
\end{align*}
\begin{align*}
(\Delta^e\otimes I)\Delta^eh(a)=&(\Delta^e\otimes I)(m\otimes h+h\otimes m)\Delta(a)\\
=&\Delta^em(a_{(1)})\otimes h(a_{(2)})+\Delta^eh(a_{(1)})\otimes m(a_{(2)})\\
=&(m\otimes m)\Delta(a_{(1)})\otimes h(a_{(2)})+(m\otimes h+h\otimes m)\Delta(a_{(1)})\otimes m(a_{(2)})\\
=&m(a_{(11)})\otimes m(a_{(12)})\otimes
h(a_{(2)})+m(a_{(11)})\otimes h(a_{(12)})\otimes
m(a_{(2)})\\&+h(a_{(11)})\otimes m(a_{(12)})\otimes m(a_{(2)})
\end{align*}
and
\begin{align*}
(I\otimes \Delta^e)\Delta^eh(a)=&(I\otimes \Delta^e)(m\otimes h+h\otimes m)\Delta(a)\\
=&m(a_{(1)})\otimes \Delta^eh(a_{(2)})+h(a_{(1)})\otimes \Delta^em(a_{(2)})\\
=&m(a_{(1)})\otimes (m\otimes h+h\otimes m)\Delta(a_{(2)})+h(a_{(1)})\otimes (m\otimes m)\Delta(a_{(2)})\\
=&m(a_{(1)})\otimes m(a_{(21)})\otimes h(a_{(22)})+m(a_{(1)})\otimes
h(a_{(21)})\otimes m(a_{(22)})\\&+h(a_{(1)})\otimes m(a_{(21)})\otimes
m(a_{(22)}).
\end{align*}
Similarly, we can prove the counitary property of $A$ is also true.
In order to prove $d$ is a coderivation of degree 1, first we have
$\varepsilon^ed^em(a)=\varepsilon^emd(a)=\varepsilon d(a)=0$ and
$\varepsilon^ed^eh(a)=\varepsilon^ehd(a)=0$, and then we should
prove $\Delta^ed^e=(d^e\otimes I+T(d^e\otimes I)T)\Delta^e$, which
follows from
\begin{align*}
(d^e\otimes I+T(d^e\otimes I)T)\Delta^em(a)&=(d^e\otimes I+T(d^e\otimes I)T)(m\otimes m)\Delta(a)\\
&=d^em(a_{(1)})\otimes m(a_{(2)})+(-1)^{|a_{(1)}|}m(a_{(1)})\otimes d^em(a_{(2)})\\
&=md(a_{(1)})\otimes m(a_{(2)})+(-1)^{|a_{(1)}|}m(a_{(1)})\otimes
md(a_{(2)}),
\end{align*}
\begin{align*}
\Delta^ed^em(a)=&(m\otimes m)\Delta d(a)\\
=&(m\otimes m)(d(a_{(1)})\otimes a_{(2)}+(-1)^{|a_{(1)}|}a_{(1)}\otimes d(a_{(2)}))\\
=&md(a_{(1)})\otimes m(a_{(2)})+(-1)^{|a_{(1)}|}m(a_{(1)})\otimes
md(a_{(2)}),
\end{align*}
\begin{align*}
\Delta^ed^eh(a)=&(m\otimes h+h\otimes m)\Delta d(a)\\
=&(m\otimes h+h\otimes m)(d(a_{(1)})\otimes a_{(2)}+(-1)^{|a_{(1)}|}a_{(1)}\otimes d(a_{(2)}))\\
=&md(a_{(1)})\otimes h(a_{(2)})+hd(a_{(1)})\otimes
m(a_{(2)})\\&+(-1)^{|a_{(1)}|}m(a_{(1)})\otimes
hd(a_{(2)})+(-1)^{|a_{(1)}|}h(a_{(1)})\otimes md(a_{(2)})
\end{align*}
and
\begin{align*}
(d^e\otimes I+T(d^e\otimes I)T)\Delta^eh(a)
=&(d^e\otimes I+T(d^e\otimes I)T)(m\otimes h+h\otimes m)\Delta(a)\\
=&d^em(a_{(1)})\otimes h(a_{(2)})+d^eh(a_{(1)})\otimes m(a_{(2)})\\
&+(-1)^{|a_{(1)}|}m(a_{(1)})\otimes d^eh(a_{(2)})+(-1)^{|a_{(1)}|}h(a_{(1)})\otimes d^em(a_{(2)})\\
=&md(a_{(1)})\otimes h(a_{(2)})+hd(a_{(1)})\otimes
m(a_{(2)})\\&+(-1)^{|a_{(1)}|}m(a_{(1)})\otimes
hd(a_{(2)})+(-1)^{|a_{(1)}|}h(a_{(1)})\otimes md(a_{(2)}).
\end{align*}
Hence $(A^e, d^e, \Delta^e, \varepsilon^e)$ is a DG coalgebra. Note
that $\Delta^e$ and $\varepsilon^e$ are DG algebra homomorphisms,
which means $(A^e, u^e, \eta^e, \Delta^e, \varepsilon^e)$ is a
graded bialgebra. Therefore, we finish the proof.
\end{proof}

In general, $A^e$ is just a DG bialgebra. But in some special cases,
$A^e$ can be endowed with the Hopf structure, such that $(A^e, u^e,
\eta^e, \Delta^e, \varepsilon^e, S^e, d^e)$ becomes a DG Hopf
algebra. Thus, we give the proof of Theorem \ref{coro36}.

\begin{proof}[Proof of Theorem 1.2]
Since the antipode $S$ is a DG Poisson homomorphism from $A$ into
$A_1$ by Lemmas \ref{lem26} and \ref{lem33}, there is a DG algebra
homomorphism $S^e: A^e\rightarrow A^{eop}$ such that $S^em=mS$ and
$S^eh=hS$ by Lemma \ref{lem25} and Proposition \ref{prop34}.

As for part (2), since $(A^e, u^e, \eta^e, \Delta^e, \varepsilon^e,
d^e)$ is a DG bialgebra by Theorem \ref{thm35}, it suffices to show
that $S^e$ is an antipode for $A^e$ if and only if $\{S(a_{(1)}),
a_{(2)}\}=0$, where $\Delta(a)=a_{(1)}\otimes a_{(2)}$ for all $a\in
A$. In fact, we have
 \begin{center}
$\eta^e\varepsilon^em(a)=\eta^e\varepsilon(a)=\varepsilon(a)1_{A^e},
~~\eta^e\varepsilon^eh(a)=0$,
\end{center}
 \begin{align*}
u^e(S^e\otimes I)\Delta^em(a)&=u^e(S^e\otimes I)(m\otimes m)\Delta(a)\\
&=S^em(a_{(1)})m(a_{(2)})=mS(a_{(1)})m(a_{(2)})\\&=m(S(a_{(1)})a_{(2)})=m(\eta\varepsilon(a))=\varepsilon(a)1_{A^e}
\end{align*}
and
 \begin{align*}
u^e(S^e\otimes I)\Delta^eh(a)=&u^e(S^e\otimes I)(m\otimes
h+h\otimes m)\Delta(a)
=S^em(a_{(1)})h(a_{(2)})+S^eh(a_{(1)})m(a_{(2)})\\=&mS(a_{(1)})h(a_{(2)})+hS(a_{(1)})m(a_{(2)})
=h(S(a_{(1)})a_{(2)})+m(\{S(a_1),
a_2\})\\=&h(\eta\varepsilon(a))+m(\{S(a_1),
a_2\})=\varepsilon(a)h(1_A)+m(\{S(a_1), a_2\})=m(\{S(a_1), a_2\}),
\end{align*}
which imply that $u^e(S^e\otimes I)\Delta^e=\eta^e\varepsilon^e$ if
and only if $\{S(a_{(1)}), a_{(2)}\}=0$, since $m$ is injective.
Note that for all $a\in A$, we have $\{a_{(1)}, S(a_{(2)})\}=0$ iff
$\{S(a_{(1)}), a_{(2)}\}=0$ by Lemma \ref{lem324}. Similarly, we can
prove the formula $u^e(I\otimes S^e)\Delta^e=\eta^e\varepsilon^e$ if
and only if $\{a_{(1)}, S(a_{(2)})\}=0$. Therefore, we finish the
proof.
\end{proof}

It is not hard to show that the following two corollaries are also true. Namely,
\begin{corollary}
Let $(A, u_A, \eta_A, \Delta_A, \varepsilon_A, S_A, \{\cdot, \cdot\}_A,
d_A)$ and $(B, u_B, \eta_B, \Delta_B, \varepsilon_B, S_B, \{\cdot, \cdot\}_B,
d_B)$ be any two DG Poisson Hopf algebras with Poisson brackets of degree $0$. Then
\begin{enumerate}
  \item $(A\otimes B)^e=A^e\otimes B^e$ is a DG bialgebra.
  \item If $\{S_A(a_{(1)}), a_{(2)}\}=0$ and $\{S_B(b_{(1)}), b_{(2)}\}=0$, where $\Delta(a)=a_{(1)}\otimes a_{(2)}, \Delta(b)=b_{(1)}\otimes b_{(2)}$ for all $a\in A$ and $b\in B$, then $(A\otimes B)^e$ is a DG Hopf algebra.
\end{enumerate}
\end{corollary}

\begin{corollary}
Let $(A, u, \eta, \Delta, \varepsilon, S, \{\cdot, \cdot\},
d)$ be a DG Poisson Hopf algebra. Then
\begin{enumerate}
  \item $(A^{op})^e=(A^e)^{op}$ is a DG bialgebra.
  \item $(A^{op})^e$ is a DG Hopf algebra if and only if $\{S(a_{(1)}), a_{(2)}\}=0$ , where $\Delta(a)=a_{(1)}\otimes a_{(2)}$ for all $a\in A$.
\end{enumerate}
\end{corollary}

\begin{exa}
Retain the notations of Example \ref{exa31}. From Theorem \ref{thm35}, we can see that the universal enveloping algebra $B^e$ is a DG bialgebra. But \[\{S(z_1), z_2\}=\{S(z), 1\}+\{S(1), z\}+\{S(-2x), y\}=2y\neq 0,\] we can't conclude that $B^e$ is a DG Hopf algebra by Theorem \ref{coro36}.
\end{exa}

\begin{exa}
Let $(L, d_L)$ be a finite dimensional DG Lie algebra over $\Bbbk$
with standard graded Lie bracket $[\cdot, \cdot]$ and let $(S(L), u,
\eta)$ be the graded symmetric algebra of $L$, i.e.,

$$S(L):=\frac{T(L)}{(a\otimes b-(-1)^{|a||b|}b\otimes a)},$$

for any $a, b\in L$, where $u$ is a multiplication and $\eta$ is a
unit. Fix a $\Bbbk$-basis $x_1, x_2, \cdots, x_n$ of $L$. Note that
$S(L)$ is the graded commutative polynomial ring $k[x_1, x_2, \cdots,
x_n]$. As in Example \ref{exa32}, we know $S(L)$ is a DG Poisson
algebra and it also has a DG Hopf structure which is compatible with
the Poisson bracket. That is, $S(L)$ is a DG Poisson Hopf algebra
with structure
\begin{center}
$\{a, b\}_{S(L)}:=[a, b], \ \ \Delta(a)=a\otimes 1+1\otimes a, \ \
\varepsilon(a)=0, \ \ S(a)=-a$,
\end{center}
for all $a, b\in L$. Observe that the universal enveloping algebra
$(S(L)^e, m, h)$ is the DG algebra generated by $x_1, \cdots, x_n,
y_1, \cdots, y_n$ subject to the relation
\begin{align*}
x_ix_j&=(-1)^{|x_i||x_j|}x_jx_i,\\
y_ix_j&=(-1)^{|x_i||x_j|}x_jy_i+\{x_i, x_j\},\\
y_iy_j&=(-1)^{|x_i||x_j|}y_jy_i+\psi(\{x_i, x_j\}),
\end{align*}
for all $i, j=1, \cdots, n$ and, $m$ and $h$ are given by $m(x_i)=x_i$
and $h(x_i)=y_i$, respectively, where $\psi: S(L)\rightarrow
S(L)\langle y_i|~i=1, \cdots, n\rangle$ is a $\Bbbk$-linear map of degree 0 defined by $\psi(x_i)=y_i$ for all $i=1, \cdots, n$. By Lemma
\ref{lem323} and Theorem \ref{coro36}, the universal enveloping
algebra $S(L)^e$ is a DG Hopf algebra with Hopf structure
 \begin{align*}
\Delta^e(x_i)&=x_i\otimes 1+1\otimes x_i
&   \Delta^e(y_i)&=y_i\otimes 1+1\otimes y_i\\
\varepsilon^e(x_i)&=0  &   \varepsilon^e(y_i)&=0\\
S^e(x_i)&=-x_i  &    S^e(y_i)&=-y_i,
\end{align*}
for all $i=1, \cdots, n$.
\end{exa}

\bigskip
\bibliographystyle{amsplain}

\end{document}